\documentclass[oneside]{amsart}
\usepackage{amssymb,amsmath,amsthm,bbm,color}
\usepackage{stackengine,geometry,wrapfig,mdframed}

\usepackage{tikz-cd}
\definecolor{orangetwo}{rgb}{0.2117647, 0.3098039, 0.5490196}
\definecolor{freading}{rgb}{ 0.8980392,1.0000000,0.9058824}
\definecolor{gray}{rgb}{0.25,0.25,0.25}
\definecolor{dgray}{rgb}{0.15,0.15,0.15}
\definecolor{white}{rgb}{1,1,1}

\definecolor{theoremback}{rgb}{0.8235294, 0.8627451, 0.9137255}
\definecolor{exampletitle}{rgb}{0.6705882, 0.7568627, 0.8705882}

\newcommand*\marcomment[3][0pt]{%
  \if l#2\reversemarginpar\def\pointer{\rightarrow}%
    \def\stackalignment{r}\fi%
  \if r#2\normalmarginpar\def\pointer{~}%
    \def\stackalignment{l}\fi%
  \marginpar{%
    \topinset{%
      \scalebox{1.1}{\textcolor{orangetwo}{$\pointer$}}}{%
      \belowbaseline[-0.5\baselineskip-#1]{%
        \stackengine%
          {-5pt}%
          {\fcolorbox{orangetwo}{white}{\parbox{4.6cm}%
		  {\vspace{3pt}\footnotesize #3}}}%
          {~\colorbox{white}{Remark}}%
          {O}%
          {l}%
          {F}%
          {F}%
          {S}%
        }%
      }{%
      3ex+#1}{%
      -2ex}%
  }%
}

\usepackage{fancyhdr}
\newtheoremstyle{nthm}
{3pt}
{3pt}
{\itshape}
{0em}
{\bfseries}
{.}
{.5em}
{}

\newtheoremstyle{ndef}
{3pt}
{3pt}
{}
{0em}
{\bfseries}
{.}
{.5em}
{}

\newtheoremstyle{nrem}
{3pt}
{3pt}
{}
{0em}
{\itshape}
{.}
{.5em}
{}

\swapnumbers
\theoremstyle{nthm}

\newmdtheoremenv[linecolor=white,linewidth=4,backgroundcolor=theoremback]{thm}[subsection]{Theorem}

\newtheorem{thm2}[subsection]{Theorem}
\newtheorem{prop}[subsection]{Proposition}

\theoremstyle{ndef}
\newtheorem{defn}[subsection]{Definition}
\theoremstyle{nrem}

\newtheorem{exmp}[subsection]{Example}
\newtheorem{rem}[subsection]{Remark}

\newcommand{\F}{\mathbb{F}}

\newcommand{\Z}{\mathbb{Z}}

\newcommand{\df}[1]{#1}

\newsavebox{\fmbox}

\DeclareMathOperator{\pfb}{pb}
\DeclareMathOperator{\Pgr}{Pgr}
\DeclareMathOperator{\Fpg}{Fpg}
\newcommand{\permterm}{fractal }
\begin{document}
\title{ The Polypermutation Group of an Associative Ring }
\begin{abstract}
     We study permutation polynomials through the device of the \emph{polypermutation group} of an associative ring $R$, denoted by $\Pgr(R)$. We derive some basic properties and compute the cardinality of $\Pgr(\Z/p^k)$ when $p\geq k$. We use this computation to determine the structure of $\Pgr(\Z/p^2)$.
\end{abstract}
\author{Jason K.C. Polak}
\date{\today}
\maketitle
\tableofcontents

\section{Introduction}
Let $p$ be a prime. Why are $\Z/p\times\Z/p$ and $\Z/p^2$ not isomorphic as rings? It's because the latter has more permutations induced by polynomials! That's what this paper is about---permutations of rings like $\Z/p^k$ induced by polynomials. We start with $R$, an associative ring with identity. We say that a polynomial $f\in R[x]$ is a \df{permutation polynomial} if the induced evaluation function $R\to R$ is bijective. For example, $x^3 + 6x^2 + x\in \Z/9[x]$ permutes $\Z/9$. For finite fields permutation polynomials have been studied extensively (cf. \cite[Chapter 7]{LidlNiederreiter2000}), and some other finite rings have also been considered. Often, research has focused on finding specific classes of permutation polynomials. In this work, we study this subject from the lesser used perspective of group theory.
\begin{defn}
    For a ring $R$, the subset of permutations of $R$ that can be represented by polynomial functions is a monoid, and we define the \df{polypermutation group} $\Pgr(R)$ of $R$ to be the subgroup generated by this monoid in the symmetric group of $R$.
\end{defn}
If $R$ is a finite field, any function $R\to R$ can be represented by a polynomial and so $\Pgr(R)$ is the symmetric group $\Delta_R$ on $R$. In \cite{Ashlock1993} and \cite{CarlitzHayes1972}, the polypermutation group  was calculated for the group ring $R[G]$ where $R$ is a finite field and $G$ is a finite abelian group. Permutation polynomials $\Z/p^k\to\Z/p^k$ have been studied before as in \cite{MR232756,MR1767431,MR764615}. Where indicated, some of the results have been obtained before, but our proofs are different. Also, we believe the presentation we give for the group $\Pgr(\Z/p^2)$ is new.

So what's in this paper? We will start with some basic properties of the polypermutation group in \S\ref{sec:basic_props}, such as the behaviour of $\Pgr(-)$ with respect to products, inverse limits, and ring homomorphisms.  For example, it turns out that $\Pgr(-)$ is a functor from the category of finite rings and surjective homomorphisms to the category of groups. The main section is \S\ref{sec:quotientIntegers}, where we compute the size of  $\Pgr(\Z/p^k)$ for a prime $p$ and an integer $k$ such that $p \geq k$:
\begin{thm2}
    Let $p$ be a prime and $k \geq 2$ be an integer such that $p \geq k$. Then
    \begin{align*}
        |\Pgr(\Z/p^k)| = p![(p-1)p^{(k^2+k-4)/2}]^p.
    \end{align*}
\end{thm2}
This result also appears in \cite{MR764615}. However, we provide a different proof which might be of some interest. Our main result is a  we compute the structure of $\Pgr(\Z/p^2)$:
\begin{thm2}
    Let $p$ be a prime and let the group $(\Z/p)^\times$ act on the group $\Z/p$ by multiplication. Let $\Delta_p$ act on the $p$-fold products $[(\Z/p)^\times]^p$ and $(\Z/p)^p$ via permuting the coordinates. Then there exists an isomorphism
    \begin{align*}
        \Pgr(\Z/p^2)\cong ( (\Z/p)^p\rtimes [(\Z/p)^\times]^p)\rtimes \Delta_p
    \end{align*}
\end{thm2}
Here, the notation $A\rtimes G$ means the semidirect product with $G$ acting on $A$. This group was also studied in \cite{MR71467}, though we believe this presentation as a semidirect product is new.

\section*{Acknowledgements} The author wishes to thank the anonymous reviewers for improving the exposition and comprehensiveness of this paper.

\section{Basic Properties}\label{sec:basic_props}
For any set $X$, we write $\Delta_X$ for the permutation group of $X$ and $\Delta_n$ for the permutation group on $n$ letters. We also write $D_n$ for the dihedral group of the regular $n$-gon if $n \geq 3$, so that $|D_n| = 2n$. By convention we set $D_2 = \Delta_2 \cong \Z/2$ and $D_1 = \{ e\}$.
\begin{prop}\label{thm:behaviourUnderProducts}
    If $R_1$ and $R_2$ are associative rings then
    \begin{align*}
        \Pgr(R_1\times R_2) \cong \Pgr(R_1)\times\Pgr(R_2).
    \end{align*}
\end{prop}
\begin{proof}
    Any polynomial permutation of $R_1\times R_2$ comes from a polynomial with coefficients in $R_1\times R_2$ and so is given by a pair of polynomials $(f_1,f_2)$ with $f_1\in R_1[x]$ and $f_2\in R_2[x]$.
\end{proof}
\begin{prop}\label{thm:projlimits}
    Let $I$ be a directed poset and consider a functor $F$ from $I$ to the category of rings.Write $R_i = F(i)$ for all $i\in I$. Suppose that for each morphism $i\to j$ in $I$, the morphism $R_i\to R_j$ with the corresponding morphism $R_i[x]\to R_j[x]$ sends permutation polynomials to permutation polynomials. Then
    \begin{align*}
        \Pgr(\varprojlim_F R_i) \cong \varprojlim_F \Pgr(R_i).
    \end{align*}
\end{prop}
\begin{proof}
   We will show that $\Pgr(\varprojlim_F R_i)$ satisfies the universal property of $\varprojlim_F \Pgr(R_i)$. To this end, let $X$ be a group and suppose we have homomorphisms $\varphi_i:X\to \Pgr(R_i)$ for each $i\in I$ such that the diagram
   \begin{equation*}
   \begin{tikzcd}
       ~ & X\ar[ld,"\varphi_i"']\ar[rd,"\varphi_j"] & ~\\
       \Pgr(R_i)\ar[rr] & ~ & \Pgr(R_j)
   \end{tikzcd}
   \end{equation*}
   commutes whenever there is a map $\Pgr(R_i)\to\Pgr(R_j)$. Thus $f_i = \varphi_i(x)$ is a permutation polynomial in $R_i[x]$ for all $i$ such that $\varphi_j(x)$ is obtained from $\varphi_i(x)$ by applying $R_i\to R_j$ if such a map exists. So, such a system of polynomials defines a polynomial $f = (f_i)$ with coefficients in $\varprojlim R_i$; we need to verify that it defines a bijection $\varprojlim R_i\to\varprojlim R_i$. Since each $f_i$ is an injection, the map $\varprojlim R_i\to\varprojlim R_i$ must certainly be an injection.

   Now suppose that $(a_i)\in \varprojlim R_i$. Since each $f_i$ is surjective, there exists an element $(b_i)\in \prod R_i$ such that $f_i(b_i) = a_i$ for each $i$. Let $\alpha:R_i\to R_j$ be the ring homomorphism in the inverse system, so that $\alpha(f_i) = f_j$ and $\alpha(a_i) = a_j$. Applying $\alpha$ to the equation $f_i(b_i) = a_i$ gives
   \begin{align*}
       f_j(\alpha(b_i)) = a_j,
   \end{align*}
and we already have $f_j(b_j) = a_j$. Since $f_j:R_j\to R_j$ is bijective, $\alpha(b_i) = b_j$ and so $(b_i)\in\varprojlim R_i$, showing that $f$ is also surjective and hence bijective. Thus we get a map $X\to \Pgr(\varprojlim R_i)$, which is a group homomorphism because each $\varphi_i$ is a group homomorphism, and by construction is the unique homomorphism that makes the appropriate diagram commute.
\end{proof}

\begin{rem}Let $R$ be a commutative ring and suppose that $R_1$ and $R_2$ are any two commutative $R$-algebras. Then there does not seem to be an easy way to determine $\Pgr(R_1\otimes_R R_2)$ from $\Pgr(R_1)$ and $\Pgr(R_2)$ as can be seen in the case of $\Z/n\otimes_\Z\Z/m\cong \Z/\gcd(m,n)$.\end{rem}
    The next proposition can be helpful when computing some polypermutation groups by hand.
\begin{prop}\label{thm:noConstantGenerator}
    Let $R$ be a ring such that every translation polynomial $x + r$ is in $\Pgr(R)$ and let $\{ f_i : i\in I\}$ be a set of generators for $\Pgr(R)$ containing all translation polynomials. Then each $f_i$ that is not a translation can be replaced by a polynomial with no constant term.
\end{prop}
\begin{proof}
    Let $f$ be in the generating set for $\Pgr(R)$. Since $f$ is a permutation, $f(r) = 0$ for some $r\in R$. Then $f(x + r)$ is also in $\Pgr(R)$, has no constant term, and may replace $f$ in the generating set.
\end{proof}

Let $R\to S$ be a ring homomorphism. When does the induced map $R[x]\to S[x]$ send permutation polynomials to permutation polynomials? This does not always happen: for example, the homomorphism $\Z/2\to \Z/2[t]$ sends $f(x) = x^2$ to a polynomial that does not induce a permutation, because $t\in \Z/2[t]$ has no square root! The first hint is a result of Rivest.

\begin{thm2}[{\cite{Rivest2001}}] Let $w\geq 2$. A polynomial $f(x) = a_0 + a_1x + \cdots + a_nx^n$ with integer coefficients reduces to a permutation polynomial in $\Z/2^w[x]$ if and only if $a_1$ is odd, $(a_2 + a_4 + \cdots )$ is even, and $(a_3 + a_5 + \cdots)$ is even. 
\end{thm2}
Since Rivest's condition on $f$ is independent of $w$, and given that such a polynomial would also reduce to a permutation polynomial in $\Z/2[x]$, we see that the reduction homomorphisms $\Z/2^k\to \Z/2^\ell$ for $\ell\leq k$ induce group homomorphisms
\begin{align*}
    \Pgr(\Z/2^k)\to \Pgr(\Z/2^\ell)
\end{align*}
which are surjective, again by Rivest's condition. It is easy to verify that when $m \mid n$, the reduction map $\Z/n\to \Z/m$ induces a map $\Pgr(\Z/n)\to\Pgr(\Z/m)$. Both of these facts are part of a more general result, which is easy to see but nevertheless useful.
\begin{prop}\label{thm:redFunctor}
    Let $R$ be a ring and $I$ an ideal of $R$. If $R/I$ is finite then the reduction modulo $I$ of any permutation polynomial $f\in R[x]$ is a permutation polynomial in $R/I[x]$. Hence, there is an induced map
    \begin{align*}
        \Pgr(R)\longrightarrow \Pgr(R/I).
    \end{align*}
\end{prop}
\begin{proof}
    Let $f\in R[x]$ be a permutation polynomial. Then its reduction modulo $I$ defines a function on $R/I$ which is surjective and hence injective because $R/I$ is a finite set.
\end{proof}
In particular, $\Pgr(-)$ is a functor from the category of finite rings and surjective morphisms to the category of finite groups. As an example use of this theorem, write 
\begin{align*}
\Z_p = \varprojlim_k \Z/p^k
\end{align*}
for the $p$-adic integers. Applying Proposition~\ref{thm:redFunctor} and Proposition~\ref{thm:projlimits}, we see that we have an isomorphism
\begin{align*}
    \Pgr(\Z_p) \cong \varprojlim_k \Pgr(\Z/p^k)
\end{align*}
showing that the monoid of polynomial permutations of $\Z_p$ is actually a group. Moreover, we can endow $\Pgr(\Z_p)$ with the topology coming from this inverse limit of finite groups, making $\Pgr(\Z_p)$ into a profinite group, though we will leave an analysis of this $p$-adic situation for a future paper. 

\begin{rem}Here is a question inspired by Proposition~\ref{thm:redFunctor} for which we do not yet have a good answer. Let $I$ be an ideal of a ring $R$ such that $R/I$ is finite. When is
    \begin{align*}\Pgr(R)\longrightarrow\Pgr(R/I)
    \end{align*}
    surjective? This is a natural question, because when it is surjective, then $\Pgr(R/I)$ would be a quotient of $\Pgr(R)$. It is certainly not always surjective, such as for $\Z\to \Z/p$ when $p$ is a prime. Even when it is surjective, lifts of permutation polynomials in $R/I[x]$ may not necessarily be permutation polynomials in $R[x]$, as in the case of $\Z/2[u]/u^2\to\Z/2$ where the polynomial $x^2\in \Z/2[u]/u^2[x]$ reduces to a permutation polynomial in $\Z/2[x]$ but does not induce a permutation of $\Z/2[u]/u^2$.\end{rem}

\section{Quotient Rings of the Integers}\label{sec:quotientIntegers}
In this section we take a look at $\Pgr(\Z/n)$. When writing permutations of $\Z/n$, we will use cycle notation with the elements labeled as $0,1,\dots,n-1$. We have already remarked that $\Pgr(R) = \Delta_R$ when $R$ is a finite field. 
\begin{prop}
     Let $n$ be squarefree with $n = p_1p_2\cdots p_k$ for distinct primes $p_1,\dots,p_k$. Then
\begin{align*}
    \Pgr(\Z/n)\cong \Delta_{p_1}\times\cdots \times\Delta_{p_k}
\end{align*}
Furthermore, if $n > 6$ then $\Pgr(\Z/n)$ is a proper subgroup of $\Delta_n$.
\end{prop}
\begin{proof}
    This follows immediately from Proposition~\ref{thm:behaviourUnderProducts}.
\end{proof}

We now consider $\Z/p^k$ where $k > 1$ and $p$ is a prime number. We start with some examples that will elucidate a method that works in principle to compute $\Pgr(R)$ for any finite ring $R$.
\begin{prop}\label{thm:pgrofzmodefour}
    The polypermutation group of $\Z/4$ is
    \begin{align*}
        \Pgr(\Z/4) \cong D_4,
    \end{align*}
    and is generated by $(0,1,2,3)$ and $(1,3)$.
\end{prop}
\begin{proof}
    The permutation $(0,1,2,3)$ can be given by the polynomial function $f(x) = x + 1$, and the permutation $(1,3)$ can be given by $f(x) = x^4 + x^2 + x$. Since $(0,1,2,3)^2 = (0,2)(1,3)$, we see that $\Pgr(\Z/4)$ contains $(0,2)$ and $(1,3)$. Now suppose $\Pgr(\Z/4)$ is not generated by $(0,1,2,3)$ and $(1,3)$. Then we need at least one more generator for $\Pgr(\Z/4)$, which we can choose by Proposition~\ref{thm:noConstantGenerator} to have no constant term. But then this generator would leave the set $\{0,2\}$ invariant, and so it would be in the subgroup $\{ e, (0,2), (1,3), (0,2)(1,3)\}$, and hence we do not need a new generator after all.
\end{proof}
Over the finite field $\F_q$, every function $\F_q\to\F_q$ can be represented by a polynomial of degree strictly less than $q$. For other finite rings, as shown by the computation of Proposition~\ref{thm:pgrofzmodefour}, there are some set endomorphisms that cannot be represented by a polynomial of any degree. Nonetheless, there are only finitely many set endomorphisms of a finite ring.
\begin{defn}
    For a finite ring $R$, we define the polynomial function bound on $R$ to be the least upper bound of the set of all $d$ such that every polynomial function $R\to R$ can be represented by a polynomial of degree at most $d$, and we write $\pfb(R)$ for this number.
\end{defn}
We can always get an upper bound for $\pfb(R)$ by computing the largest integer $d$ such that the polynomial functions $x,x^2,\dots,x^d$ are all distinct; then $\pfb(R) \leq d$. For example, by this method we see that $\pfb(\Z/9)\leq 7$ and $\pfb(\Z/27)\leq 20$. Since there are only finitely many ring structures on a finite set of a given cardinality, one should be able to express this bound in terms of this cardinality. Using this number, we can compute $\Pgr(R)$ for any finite ring. Since this may be the only method for some finite rings, we illustrate it with an example, using Sage to avoid lengthy hand-computations.
\begin{exmp}\label{exmp:zm8}We have
    \begin{align*}
        \Pgr(\Z/8)\cong (\Z/2)^4\rtimes D_4.
    \end{align*}
    Indeed, we first compute powers of elements in $\Z/8$, which gives us $\pfb(\Z/8)\leq 4$. We need three permutations to generate $\Pgr(\Z/8)$: the permutation $(0, 1, 2, 3, 4, 5, 6, 7)$ given by $f(x) = x + 1$, the permutaiton $(1, 3, 5, 7)(2,6)$ given by $f(x) = x^4 + x^2 + x$, and the permutation $(1,5)$ given by $f(x) = x^4 + x^2 + 3x$.

    These permutations generate the group $\Pgr(\Z/8)$ which has order~128. It has a normal subgroup isomorphic to $(\Z/2)^4$ fitting into an exact sequence
    \begin{align*}
        1\to (\Z/2)^4\to \Pgr(\Z/8)\to D_4\to 1.
    \end{align*}
    The subgroup isomorphic to $(\Z/2)^4$ can be generated by the set $\{ (3,7), (2,6), (1,5), (0,4)\}$, and the quotient $D_4$ has coset representatives
    \begin{align*}
        \{ e, (1,3)(5,7), (0,1)(2,3)(4,5)(6,7), (0,1,2,3)(4,5,6,7), (0,2)(4,6),\\
        (0,2)(1,3)(4,6)(5,7), (0,3,2,1)(4,7,6,5), (0,3)(1,2)(4,7)(5,6) \}\subseteq H.
    \end{align*}
    The dihedral group $D_4$ also has the presentation
    \begin{align*}
        D_4 = \langle~ r,s ~|~ r^4, s^2, r^ks = sr^{-k} ~\rangle
    \end{align*}
    and an isomorphism to the quotient of $\Pgr(\Z/8)$ is given by
    \begin{align*}
        r &\longmapsto (0,1,2,3)(4,5,6,7)\\
        s &\longmapsto (1,3)(5,7)
    \end{align*}
    Similarly, an embedding of $(\Z/2)^4$ into $\Pgr(\Z/8)$ is given by
    \begin{align*}
        (1,0,0,0) &\longmapsto (2,6)\\
        (0,1,0,0) &\longmapsto (3,7)\\
        (0,0,1,0) &\longmapsto (0,4)\\
        (0,0,0,1) &\longmapsto (1,5).
    \end{align*}
    The intuition is to think of a square whose vertices are labeled by $(2,6), (3,7), (0,4)$ and $(1,5)$ around going either clockwise or counterclockwise. Using these isomorphisms, the action of $D_4$ on $(\Z/2)^4$ is given on generators by
    \begin{align*}
        r*(a,b,c,d) &= (d,a,b,c)\\
        s*(a,b,c,d) &= (a,c,b,d)
    \end{align*}
and it gives an explicit isomorphism
\begin{align*}
    \Pgr(\Z/8)\cong (\Z/2)^4\rtimes D_4.
\end{align*}
\end{exmp}
These techniques can be used for any finite ring but for larger cardinalities, the computations quickly become prohibitive. Next, we will derive a few results necessary to compute the cardinality of $\Pgr(\Z/p^k)$. We first note that if $f\in \Z/p^k[x]$, not necessarily a permutation polynomial, then
\begin{align}\label{eqn:extenderRule}
    f(x + mp) = f(x) + mpf'(x) + (mp)^2f''(x) + \cdots + (mp)^{k-1}f^{(k-1)}(x)
\end{align}
for all $x\in \Z/p^k$ and where $f'$ denotes the formal derivative of $f$. Therefore, $f$ is actually determined by a choice of $0,1,\dots,p-1$ and a choice of derivatives $f^{(i)}(0),\dots,f^{(i)}(p-1)$ for $i=0,\dots,k-1$. Can any such choice be represented by a polynomial? This is the content of a theorem of Carlitz.
\begin{thm2}[{\cite[Theorem 3]{Carlitz1964}}]
    A function $f:\Z/p^k\to\Z/p^k$ can be represented by a polynomial in $\Z/p^k[x]$ if and only if
    \begin{align*}
        f(x + mp) = f_0(x) + mpf_1(x) + \cdots + (mp)^{k-1}f_{k-1}(x)
    \end{align*}
    for all $m$ and $x = 0,\dots,p-1$ where each $f_i:\Z/p\to \Z/p^k$ is an arbitrary function.
\end{thm2}
So any function $f:\Z/p^k\to\Z/p^k$ obtained by choosing $f^{(i)}(0),\dots,f^{(i)}(p-1)$ for $i = 0,\dots,k-1$ and extending by Equation~\eqref{eqn:extenderRule} can also be defined by a polynomial in $\Z/p^k[x]$.
\begin{prop}
    A function $f:\Z/p^k\to\Z/p^k$ obtained by the method just described is a permutation of $\Z/p^k$ if and only if $f(0),\dots,f(p-1)$ are all distinct modulo $p$ and $f'(0),\dots,f'(p-1)\in (\Z/p^k)^\times$.
\end{prop}
\begin{proof}
    Suppose there exists an $x\in \{0,\dots,p-1\}$ such that $f'(x) \in (p)$. Choosing $m = p^{k-2}$, we see that
    \begin{align*}
        f(x + mp) = f(x) + p^{k-1}f'(x) = f(x) \in \Z/p^k.
    \end{align*}
    Therefore, the conditions: $f(0),\dots,f(p-1)$ are all distinct modulo $p$ and $f'(0),\dots,f'(p-1)\in (\Z/p^k)^\times$ is certainly necessary for the corresponding function to be a permutation.
    
    Now we prove sufficiency. Since $f(x + mp)$ and $f(x)$ are the same modulo $p$, it suffices to fix an $x$ and show that $f(x), f(x + p), \dots, f(x + (p^{k-1}-1)p)$ are all distinct. Suppose not. Then there exists distinct $m_1,m_2\in 0,1,\dots,p^{k-1} - 1$ such that
    \begin{align*}
        f(x + m_1p) = f(x + m_2p).
    \end{align*}
    Then by Equation~\eqref{eqn:extenderRule}, we must have
    \begin{align*}
        0 = p(m_1 - m_2)f'(x) + p^2(m_1^2 - m_2^2)f''(x) + \cdots + p^{k-1}(m_1^{k-1} - m_2^{k-1})f^{(k-1)}(x).
    \end{align*}
    Reducing modulo $p^2$ we see that $m_1 - m_2\in (p)$. But then $m_1^2 - m_2^2\in (p)$ and so reducing modulo $p^3$ we see that $m_1 - m_2\in (p^2)$. Continuing along this fashion, using that $m_1^\ell - m_2^\ell$ has $m_1 - m_2$ as a factor, we can conclude that $m_1 - m_2\in (p^{k-1})$. But we have chosen $m_1,m_2\in \{0,\dots,p^{k-1}-1\}$, and so $m_1 = m_2$. Thus $f$ is indeed a permutation.
\end{proof}
Looking at Equation~\eqref{eqn:extenderRule} again, we see that to obtain any permutation it suffices to choose $f(0),\dots,f(p-1)$, exactly one from each coset of the ideal $(p)$ in $\Z/p^k$, an ordering of these cosets, and for $x = 0,1,\dots,p$ the elements $f'(x)\in (\Z/p^{k-1})^\times$, and $f^{(\ell)}(x) \in \Z/p^{k-\ell}$ for $\ell > 1$. It is easy to see that this gives
\begin{align*}
    p!(p^{k-1})^p[p^{k-2}(p-1)]^p[p^{k-2}p^{k-3}\cdots p]^p = p![(p-1)p^{(k^2+k-4)/2}]^p
\end{align*}
many choices. Moreover, we consider all of these choices to be elements of $\Z/p^k$ through the set inclusion (not ring homomorphism!) $\Z/p^\ell\to\Z/p^k$  defined by $n\mapsto n$ for $\ell \leq k$; this is to avoid writing the more cumbersome $0,1,\dots,p^\ell\in \Z/p^k$. Do different choices necessarily lead to different permutations? Not necessarily. We have to impose one additional condition for this to be so and this is the content of the next result.
\begin{thm2}\label{thm:uniqueCount}
    Let $p$ be a prime and $k \geq 2$ be an integer such that $p \geq k$. Then
    \begin{align*}
        |\Pgr(\Z/p^k)| = p![(p-1)p^{(k^2+k-4)/2}]^p.
    \end{align*}
\end{thm2}
\begin{proof}
    Consider two permutations $f$ and $g$ defined by the aforementioned choices of $f^{(i)}(x)$ for $i = 0,1,\dots,k-1$ and $x = 0,1,\dots,p-1$. Let us suppose that $f$ and $g$ induce the same permutation on $\Z/p^k$. To prove the theorem we must show that $f^{(i)}(x) = g^{(i)}(x)$ for all $i = 0,\dots,k-1$ and all $x = 0,1,\dots,p-1$ or equivalently, that $d_i = f^{(i)}(x) - g^{(i)}(x) \in \Z/p^k$ is zero for all $i$.
    
    Since $f$ and $g$ are supposed to be the same, the difference of Equation~\eqref{eqn:extenderRule} for $f$ and the analogue for $g$ gives the identity in $\Z/p^k$:
    \begin{align*}
        0 = mpd_1 + (mp)^2d_2 + \cdots + (mp)^{k-1}d_{k-1}
    \end{align*}
    that holds for all $m$. Let $m_1,\dots,m_{k-1}$ be arbitrary. For each $m_i$, we obtain an identity, all of which collectively can be expressed in matrix notation:
    \begin{align}\label{eqn:pfmatrixId}
    0 = 
    \begin{bmatrix}
        m_1 & m_1^2 & \cdots & m_1^{k-1}\\
        m_2 & m_2^2 & \cdots & m_2^{k-1}\\
        \vdots & \vdots & \ddots & \vdots\\
        m_{k-1} & m_{k-1}^2 & \cdots & m_{k-1}^{k-1}
    \end{bmatrix}
    \begin{bmatrix}
        pd_1\\
        p^2d_2\\
        \vdots\\
        p^{k-1}d_{k-1}.
    \end{bmatrix}
\end{align}
The determinant of the matrix with $i,j$-entry $m_i^j$ is just a variant of the Vandermonde matrix; its determinant is
\begin{align*}
    \prod_i m_i \prod_{i > j} (m_i - m_j).
\end{align*}
The identity in~\eqref{eqn:pfmatrixId} shows that this determinant annihilates $p^\ell d_\ell$ in the ring $\Z/p^k$. Because we have assumed that $p \geq k$, we can choose $m_1,\dots,m_{k-1}$ in the set $\{1,2,\dots,p-1\}$ all \emph{distinct}, and so consequently $p^\ell d_\ell = 0$. But $d_\ell\in \{0,1,2,\dots,p^{k-\ell} - 1\}$ and hence $d_\ell = 0$.
\end{proof}
\begin{rem}
    The same proof idea does not seem to work for $p < k$ because the determinant annihilating $p^\ell d_\ell$ will only give a lower bound in this case. In fact, we recall Example~\ref{exmp:zm8} that $|\Pgr(\Z/2^3)| = 2^7$, whereas putting $p=2$ and $k=3$ in Theorem~\ref{thm:uniqueCount} gives the number $2^9$. 
\end{rem}

\begin{thm2}
    Let $p$ be a prime and let the group $(\Z/p)^\times$ act on the group $\Z/p$ by multiplication. Let $\Delta_p$ act on the $p$-fold products $[(\Z/p^)\times]^p$ and $(\Z/p)^p$ via permuting the coordinates. Then there exists an isomorphism
    \begin{align*}
        \Pgr(\Z/p^2)\cong ( (\Z/p)^p\rtimes [(\Z/p)^\times]^p)\rtimes \Delta_p
    \end{align*}
\end{thm2}
\begin{proof}
    For a polynomial permutation $f$ on $\Z/p^k$, let $\sigma_f$ be the permutation that $f$ induces on $\Z/p$. By our previous discussion, to give $f$ is the same thing as to give $\sigma_f$, elements $a_0,\dots,a_{p-1}\in \Z/p$, and elements $f_0,\dots,f_{p-1}\in(\Z/p)^\times$, which defines $f$ by the conditions that
    \begin{equation}\label{eqn:detforms}
        \begin{split}
        f(x) &= \sigma_f(x) + a_xp \\
        f(x + mp) &= f(x) + mpf_x
    \end{split}
    \end{equation}
    for $x = 0,\dots,p-1$ and $m$ arbitrary. Then it follows that $f(x + mp) = f(x) + mpf_y$ where $y\in \{ 0,\dots,p-1\}$ and $y\equiv x\pmod{p}$. Consider the map:
    \begin{align*}
        \Pgr(\Z/p^2)&\longrightarrow( (\Z/p)^p\rtimes [(\Z/p)^\times]^p)\rtimes \Delta_p\\
        f&\longmapsto \bigl( (a_0,\dots,a_{p-1}), (f_0,\dots,f_{p-1}), \sigma_f \bigr)
    \end{align*}
    Theorem~\ref{thm:uniqueCount} shows that this map is a bijection. To show that it is a homomorphism we compute the product of two elements in the iterated semidirect product: Suppose $g$ is another polynomial permutation defined by constants $b_i\in \Z/p, g_i\in (\Z/p)^\times$, and $\sigma_g$. Then
    \begin{align*}
        [ (a_i), (f_i), \sigma_f ][ (b_i), (g_i), \sigma_g ] &= ( \sigma_f*( (b_i), (g_i) ) + (a_i, f_i), \sigma_g\circ\sigma_f )\\
        &= [ ( (b_{\sigma_f(i)}), (g_{\sigma_f(i)}))( (a_i), (f_i) ), \sigma_g\circ\sigma_f]\\
        &= [ ( b_{\sigma_f(i)} + g_{\sigma_f(i)}a_i), (g_{\sigma_f(i)}f_i), \sigma_g\circ\sigma_f ].
    \end{align*}
    On the other hand, by directly using the formulas in~\eqref{eqn:detforms}, we see that:
\begin{align*}
    (g\circ f)(i) &= g(\sigma_f(i) + a_ip)\\
    &= (\sigma_g\circ\sigma_f)(i) + p(b_{\sigma_f(i)} + a_ig_{\sigma_f(i)}).
\end{align*}
and $(g\circ f)(i + mp) = (g\circ f)(i) + mpf_ig_{\sigma_f(i)}$.
\end{proof}
The same proof will not work for $k > 2$; the problem is that the structure of the group $\Pgr(\Z/p^k)$ is more complicated and it is not clear to the author if there is any nice presentation of it. Nonetheless, we emphasize that with Theorem~\ref{thm:uniqueCount}, it is possible to write a fairly fast algorithm that will determine all the generators of $\Pgr(\Z/p^k)$ for any $k$ as a subgroup of $\Delta_{p^k}$. 

\begin{rem}A polynomial permutation of $\Z/p^k$ also induces a permutation of $\Z/p^\ell$ for $\ell=1,\dots,k$ by Proposition~\ref{thm:redFunctor}. Inspired by this fact, it is tempting to introduce the following definition: Let $R$ be a commutative ring and $I$ an ideal of $R$. We say that a permutation of $R/I^k$ is an \df{$I$-\permterm permutation} of $R/I^k$ if it induces permutations of $R/I^\ell$ for $\ell=1,\dots,k$. If the ideal is understood, we simply say \df{\permterm permutation}. Let us write $\Fpg_I(R/I^k)$ for the group of $I$-fractal permutations of $R/I^k$. 
    
    As we have said, permutation polynomials in $\Z/p^k[x]$ induce \permterm permutations of $\Z/p^k$. However, the converse is false in general! Indeed, the following fractal permutation of $\Z/27$ is not given by any polynomial:
\begin{align*}
    (0,5)(1,13,7,10,4,25)(2,15,8,3,11,24,17,21,20,6,26,12)(9,14,18,23)(16,19,22)
\end{align*}
In fact $\Pgr(\Z/p^k)$ is usually a proper subgroup of $\Fpg_p(\Z/p^k)$. Now, using the notion of fractal permutation, we can define the $I$-fractal permutation group of $R$, or the fractal permutation group of the pair $(R,I)$ as the limit
\begin{align*}
    \Fpg_I(R) := \varprojlim_k \Fpg_I(R/I^k).
\end{align*}
The structure and meaning of this group are still mysterious, but we leave this for future research.
\end{rem}

\bibliographystyle{unsrt}

\end{document}